%% file: main.tex
\documentclass[11pt]{article}

\usepackage{pedro}
\usepackage{scalerel}

\newcommand{\TODO}[1]{}

\begin{document}
\title{Sharp Inner Product Correlations for Hypercube Bijections}

\author{
  Ijay Narang\thanks{Georgia Institute of Technology. \texttt{inarang3@gatech.edu}.}
  \and
  Muchen Ju\thanks{Fudan University. \texttt{22300180095@m.fudan.edu.cn}.}
}

\date{July 25, 2025}
\maketitle

\begin{abstract}
We resolve a conjecture of Rob Morris concerning bijections on the hypercube. Specifically, we show that for any bijection \(f : \{-1,1\}^n \to \{-1,1\}^n\),
\[
\Pr_{x,y \in \{-1,1\}^n}\big[ \langle x,y \rangle \ge 0 \;\text{and}\; \langle f(x),f(y) \rangle \ge 0 \big] 
\;\;\ge\; \tfrac{1}{4} - O(1/\sqrt{n}),
\]
implying the same lower bound for the joint event under any two bijections. Our proof proceeds by applying the spectral decomposition of the Hamming association scheme, which allows us to reformulate the problem as a linear program over the Birkhoff polytope. This makes it possible to isolate the contribution of the nontrivial spectrum, which we show is asymptotically negligible, leaving the dominant contribution arising from the principal eigenvalue.

\end{abstract}

\input{intro}

\input{inner_product}
\input{acknowledgements}

\bibliographystyle{alpha}
\bibliography{ref}

\end{document}

%% file: intro.tex
\section{Introduction}

The discrete hypercube \(\{-1,1\}^n\) is one of the central objects in combinatorics, discrete geometry, and the analysis of Boolean functions. Its algebraic and metric structure makes it a natural setting for studying extremal and probabilistic questions involving symmetry. In this paper, we study how much a bijection of the hypercube can distort the event that two random points have nonnegative inner product.

More specifically, for bijections \(f_1,f_2:\{-1,1\}^n\to\{-1,1\}^n\), consider
\[
\Pr_{x,y}\!\left[\langle f_1(x),f_1(y)\rangle\ge 0 \ \text{and}\ \langle f_2(x),f_2(y)\rangle\ge 0\right].
\]
If \(f_1\) and \(f_2\) are chosen independently and uniformly at random, then this probability is known to be $\frac14+o(1).$ Its maximum is also trivial, attained when \(f_1=f_2=\mathrm{id}\). The main question of this paper is therefore to understand the minimum over all bijections. Observe that since the above quantity depends only on the composition \(f_2\circ f_1^{-1}\), this reduces to the case of comparing an arbitrary bijection $f$ with the identity.

A natural conjecture is that the random-bijection value is in fact asymptotically extremal. In other words, one may ask whether every bijection \(f:\{-1,1\}^n\to\{-1,1\}^n\) satisfies
\[
\Pr_{x,y}\!\left[\langle x,y\rangle\ge 0 \ \text{and}\ \langle f(x),f(y)\rangle\ge 0\right]\ge \frac14-o(1).
\]
At PCMI 2025, Rob Morris conjectured that this is indeed the case. In this paper, we resolve this conjecture in the affirmative. In particular, we show

\begin{theorem} 
 \label{thm:main}
Let \( f : \{-1,1\}^n \to \{-1,1\}^n \) be any bijection on the hypercube. Then,
\[
\mathbb{P}_{x,y \in \{-1,1\}^n} \left[ \langle x, y \rangle \ge 0 \;\;\text{and}\;\; \langle f(x), f(y) \rangle \ge 0 \right]
\ge \frac{1}{4} - O(\frac{1}{\sqrt{n}}),
\]
where the probability is over uniformly random independent \( x, y \in \{-1,1\}^n \).
\end{theorem}

Our proof is spectral. We rewrite the probability in question as an optimization problem over permutation matrices, diagonalize the relevant kernel using the Hamming association scheme, and then pass to a linear-programming relaxation over the Birkhoff polytope. This isolates the contribution of the principal eigenspace, which yields the main term \(1/4\), and reduces the remainder to an analytic estimate on the nontrivial spectrum.

Beyond being a natural question to ask, Theorem \ref{thm:main} is also motivated by recent work of Balister et al (\cite{balister2024upperboundsmulticolourramsey}), which proved that the $r$-colour Ramsey number $R_r(k)$, defined as the minimum $n \in \mathbb{N}$ such that every $r$-colouring of the edges of the complete graph $K_n$ contains a monochromatic copy of $K_k$, satisfies  $R_r(k) \le e^{-\delta k^r} \, r^k$ for each fixed $r > 2$, where $\delta = \delta(r) > 0$ and $k$ is sufficiently large.

Key to their proof is the following geometric lemma:

\begin{lemma}[Lemma 3.1 in \cite{balister2024upperboundsmulticolourramsey}]
 Fix $n,r \in \mathbb{N}$, define $\beta = 3^{-4r}$ and $C = 4r^{3/2}$. Let $U$ and $U'$ be i.i.d.\ random variables taking values in a finite set $X$, and
let $\sigma_1, \ldots, \sigma_r : X \to \mathbb{R}^n$ be arbitrary functions.
There exists $\lambda \geq -1$ and $i \in [r]$ such that
\[
\Pr\!\left[
\langle \sigma_i(U), \sigma_i(U') \rangle \geq \lambda
\ \text{and}\
\langle \sigma_j(U), \sigma_j(U') \rangle \geq -1 \ \text{for all } j \ne i
\right]
\geq \beta e^{-C\sqrt{\lambda+1}}.
\]
\end{lemma}

Take $X$ to be the hypercube embedded in $\mathbb{R}^n$, $\sigma_i = f_i$ to be bijections on the hypercube, and note that  
\begin{align*}
&\mathbb{P}_{x,y \in \{-1,1\}^n} \left[\lor_{i=1}^{r} \langle f_i(x), f_i(y) \rangle = -1 \right] \\
 \leq & \sum_{i=1}^{r}\mathbb{P} \left[\langle f_i(x), f_i(y) \rangle = -1 \right] \\
 = &r\mathbb{P}_{x,y \in \{-1,1\}^n} \left[ \langle x, y \rangle = -1 \right] \\
 = &r\mathbb{P}_{x,y \in \{-1,1\}^n} \left[ y \text{ differ from } x \text{ in exactly } \frac{n+1}{2} \text{ coordinates}\right] \\
 \leq &r\frac{\binom{{\lfloor \frac{n}{2}\rfloor}}{n}}{2^{n}} =O(\frac{r}{\sqrt{n}})
\end{align*}
their result yields that for all collections of bijections $f_1, f_2, \cdots,f_r : \{-1, 1\}^n \rightarrow \{-1, 1\}^n$, \[\Pr_{x,y \in \{-1,1\}^n} \left[ \land_{i=1}^r \langle f_i(x), f_i(y) \rangle \ge 0 \right] \ge C^{\prime}-O(\frac{r}{\sqrt{n}})\] for some constant $C^{\prime} > 0$.

This places \Cref{thm:main} as a sharp asymptotic result for the $r = 2$ case. Before we present the full details of the proof, we discuss relevant background on the Hamming association scheme and the Birkhoff polytope.

%% file: inner_product.tex
\section{Preliminaries} \label{sec:prelim}
\subsection{Hamming Association Scheme}
Let \( \Omega = \{0,1\}^n \) denote the hypercube. The Hamming association scheme on \( \Omega \) is generated by the symmetric matrices \( A_0, A_1, \dots, A_n \in \mathbb{R}^{2^n \times 2^n} \), where each \( A_k \) is defined entrywise by
\[
(A_k)_{x,y} = \mathbf{1}[d(x,y) = k],
\]
with \( d(x,y) \) denoting the Hamming distance. The matrices \( A_k \) commute and are simultaneously diagonalizable; together they span the Bose–Mesner algebra of the scheme. The Hamming Association Scheme has been well-studied and has numerous applications in coding theory, for a more detailed treatment (see e.g. Chapter 10 of \cite{godsil2016}).
\\\\
\noindent A common orthonormal eigenbasis for all \( A_k \) is given by the Walsh–Hadamard vectors \( \bm{\chi}_S \in \mathbb{R}^{2^n} \), indexed by subsets \( S \subseteq [n] \). Note that these are defined by evaluations of the Fourier Characters:
\[
\bm{\chi}_S(x) := (-1)^{\sum_{i \in S} x_i}, \quad x \in \Omega,
\]
and we treat \( \bm{\chi}_S \) as a vector in \( \mathbb{R}^{2^n} \) whose coordinates are indexed by \( x \in \Omega \). These vectors form an orthonormal basis under the normalized inner product
\[
\langle \bm{f}, \bm{g} \rangle := \frac{1}{2^n} \sum_{x \in \Omega} f(x) g(x).
\]

Each matrix \( A_k \) acts diagonally in this basis. Specifically, the Walsh–Hadamard vector \( \bm{\chi}_S \) is an eigenvector of \( A_k \) with eigenvalue
\[
A_k \bm{\chi}_S = K_k(|S|) \bm{\chi}_S,
\]
where \( |S| \) is the cardinality of the subset \( S \), and \( K_k(i) \) is the degree-\( k \) Krawtchouk polynomial evaluated at \( i \), given by
\[
K_k(i) = \sum_{j=0}^k (-1)^j \binom{i}{j} \binom{n - i}{k - j}.
\]

Letting \( U \in \mathbb{R}^{2^n \times 2^n} \) denote the orthogonal matrix whose columns are the normalized Walsh–Hadamard vectors \( \bm{\chi}_S / \sqrt{2^n} \), we obtain the spectral decomposition
\[
A_k = U \Lambda_k U^\top,
\]
where \( \Lambda_k \in \mathbb{R}^{2^n \times 2^n} \) is diagonal with entries \( (\Lambda_k)_{S,S} = K_k(|S|) \). 

\subsection{LP Geometry and the Birkhoff Polytope}

A linear program (LP) optimizes a linear objective over a polyhedron defined by linear constraints. The standard form of a linear program is:
\[
\begin{aligned}
\text{minimize} \quad & c^\top x \\
\text{subject to} \quad & Ax = b, \\
& x \ge 0,
\end{aligned}
\]
where \( c \in \mathbb{R}^d \), \( A \in \mathbb{R}^{m \times d} \), and \( b \in \mathbb{R}^m \). The feasible region \( \mathcal{P} = \{ x \in \mathbb{R}^d : Ax = b,\ x \ge 0 \} \) is a convex polytope when bounded.

\begin{fact}[Extreme Point Principle] \label{fact: epp}
Let \( \mathcal{P} \subseteq \mathbb{R}^d \) be a polytope and \( c \in \mathbb{R}^d \). Then the minimum of the linear program \( \min_{x \in \mathcal{P}} c^\top x \) is attained at an extreme point (vertex) of \( \mathcal{P} \).
\end{fact}

\noindent In our setting, the relevant polytope is the Birkhoff polytope, which is defined as the set of all doubly stochastic matrices. More formally, we denote this polytope by
 \( \mathcal{D}_n \subseteq \mathbb{R}^{n \times n} \) where

\[
\mathcal{D}_n := \left\{ X \in \mathbb{R}^{n \times n} \;\middle|\; X_{ij} \ge 0,\ \sum_{j=1}^n X_{ij} = 1,\ \sum_{i=1}^n X_{ij} = 1 \right\}.
\]

The Birkhoff-von Neumann theorem (Theorem 8.6 of \cite{schrijver1986theory}) gives the following useful characterization of $\mathcal{D}_n$.

\begin{fact} \label{fact:bvn_thm}
The vertices of \( \mathcal{D}_n \) are exactly the set of \( n \times n \) permutation matrices.
\end{fact}

\noindent Note that the result implies that linear optimization over \( \mathcal{D}_n \) reduces to optimization over permutations, which is critical to our proof.

\section{Proof of \Cref{thm:main}} \label{sec:thm}
We begin by translating
\begin{equation} \label{eq:probability}
\Pr\left[ \langle x, y \rangle \ge 0 \;\text{and}\; \langle f(x), f(y) \rangle \ge 0 \right],
\end{equation}
into the language of linear algebra, which will allow us to establish a universal lower bound on this probability for all bijections \( f \). We fix an enumeration of the hypercube: let \( x_1, \dots, x_N \in \{-1,1\}^n \), where \( N = 2^n \), represent all possible hypercube vectors. Define the matrix \( M \in \mathbb{R}^{N \times N} \) by
\[
M_{i,j} = \mathbf{1}\left[\langle x_i, x_j \rangle \ge 0\right],
\]
so that \( M \) encodes whether the inner product between two hypercube vectors is nonnegative. The bijection \( f \) induces a permutation matrix \( P \in \mathbb{R}^{N \times N} \) such that \( P_{ij} = \mathbf{1}[f(x_i) = x_j] \).
\\\\
\noindent Then the condition $\langle x_i, x_j \rangle \geq 0$ and $\langle f(x), f(y) \rangle \geq 0$ holds if and only if \( M_{i,j} = M_{f(i),f(j)} = 1 \) for a random pair \( i,j \). This yields that:
\[
\Pr\left[ \langle x, y \rangle \ge 0 \text{ and } \langle f(x), f(y) \rangle \ge 0 \right] = \frac{1}{N^2} \sum_{i,j=1}^N M_{i,j} \cdot M_{f(i),f(j)} = \frac{1}{N^2} \operatorname{Tr}(P^\top M P M).
\]

\noindent Thus, lower bounding \eqref{eq:probability} over all bijections is equivalent to lower bounding
\begin{equation} \label{eq:optimization}
\min_{P \in \mathcal{P}} \frac{1}{N^2} \operatorname{Tr}(P^\top M P M),
\end{equation}
where \( \mathcal{P} \) denotes the set of \( N \times N \) permutation matrices.
\\\\
\noindent To analyze \eqref{eq:optimization}, we first fully characterize the spectrum of the matrix $M$ via its relation to the Hamming Association Scheme. This is done in the following Lemma:

\begin{lemma} \label{lem:inner_prod_spectrum}
Let \( M \in \mathbb{R}^{2^n \times 2^n} \) be the matrix with entries given by $M_{x,y} = \mathbf{1}\left[\langle x, y \rangle \ge 0\right]$,
for \( x, y \in \{-1,1\}^n \). Then \( M \) is diagonalized by the Walsh–Hadamard basis, with eigenvalues (indexed by subsets $S \subseteq [n]$)
\[
\lambda_S = \sum_{d=0}^{\lfloor n/2 \rfloor} K_d(|S|),
\]
where \( K_d \) is the degree-\( d \) Krawtchouk polynomial. Each \( \lambda_S \) has multiplicity \( \binom{n}{|S|} \).
\end{lemma}

\begin{proof}
For \( x, y \in \{-1,1\}^n \), the inner product satisfies \( \langle x, y \rangle = n - 2d_H(x, y) \). Thus,
\[
M_{x,y} = \mathbf{1}[\langle x, y \rangle \ge 0] = \mathbf{1}[d_H(x, y) \le n/2],
\]
so we may write
\[
M = \sum_{d = 0}^{\lfloor n/2 \rfloor} A_d = U \left( \sum_{d=0}^{\lfloor n/2 \rfloor} \Lambda_d \right) U^\top = U \Lambda U^\top
\]
where $A_d$ corresponds to the distance $d$ matrix of the Hamming Association Sceheme. The result follows from the fact that $\Lambda$ is diagonal with 
$\Lambda_{S, S} = \sum_{d=0}^{\lfloor n/2 \rfloor} K_d(|S|)$.
\end{proof}

Applying \Cref{lem:inner_prod_spectrum} to \eqref{eq:optimization}, we have:

\begin{align}
\frac{1}{N^2} \operatorname{Tr}(P^\top M P M)
&= \frac{1}{N^2} \operatorname{Tr}(P^\top U \Lambda U^\top P U \Lambda U^\top) \notag \\
&= \frac{1}{N^2} \operatorname{Tr}\left((U^\top P^\top U) \Lambda (U^\top P U) \Lambda\right) \notag \\
&= \frac{1}{N^2} \sum_{S,T \subseteq [n]} \lambda_S \lambda_T \left(U^\top P U\right)_{S,T}^2. \label{eq:perm-diagonal}
\end{align}

Furthermore, a direct implication of \Cref{lem:inner_prod_spectrum} of is the following
\[
\lambda_{\emptyset} = \sum_{k = 0}^{\lfloor n/2 \rfloor} K_k(0) = \sum_{k = 0}^{\lfloor n/2 \rfloor} \binom{n}{k} \ge 2^{n-1}
\]

Note that the bulk of the spectrum of $M$ lies in its largest eigenvalue. As such, to understand the above sum \eqref{eq:perm-diagonal}, we require an understanding of the coefficients that correspond to the terms $\lambda_S\lambda_T$ where one or both of $S$ and $T$ equal the empty set. This is given by the following lemma.

\begin{lemma} \label{lem:matrix_fact}
Let \( U \in \mathbb{R}^{N \times N} \), with \( N = 2^n \), be the orthogonal Walsh--Hadamard matrix whose columns are \( \bm{\chi}_S / \sqrt{N} \) for \( S \subseteq [n] \), and let \( P \in \mathbb{R}^{N \times N} \) be any permutation matrix. If $R := U^\top P U$, then:
\begin{align*}
R_{\emptyset,\emptyset} &= 1, \\
R_{\emptyset,T} &= R_{T,\emptyset} = 0 \quad \text{for all } T \ne \emptyset.
\end{align*}
\end{lemma}

\begin{proof}
Since \( U_{x,S} = \frac{1}{\sqrt{N}} (-1)^{\langle S, x \rangle} \), we have:
\[
R_{\emptyset,T} = \sum_{a,b} U_{a,\emptyset} P_{a,b} U_{b,T}
= \sum_{a,b} \frac{1}{\sqrt{N}} P_{a,b} \cdot \frac{1}{\sqrt{N}} (-1)^{\langle T, x_b \rangle}
= \frac{1}{N} \sum_{b} (-1)^{\langle T, x_b \rangle},
\]
where we used that \( \sum_a P_{a,b} = 1 \) since \( P \) is a permutation matrix.

If \( T \ne \emptyset \), then the character \( x \mapsto (-1)^{\langle T, x \rangle} \) is nontrivial and satisfies:
\[
\sum_{x \in \{0,1\}^n} (-1)^{\langle T, x \rangle} = 0.
\]
Thus, \( R_{\emptyset,T} = 0 \). The same argument shows \( R_{T,\emptyset} = 0 \) by symmetry.

Finally, for \( R_{\emptyset,\emptyset} \), we compute:
\[
R_{\emptyset,\emptyset} = \sum_{a,b} U_{a,\emptyset} P_{a,b} U_{b,\emptyset}
= \sum_{a,b} \frac{1}{\sqrt{N}} P_{a,b} \cdot \frac{1}{\sqrt{N}} = \frac{1}{N} \sum_{a,b} P_{a,b} = 1,
\]
since \( P \) is a permutation matrix and hence has exactly one 1 in each row and column.
\end{proof}

Applying \Cref{lem:matrix_fact}, we have that:

\begin{align*}
\eqref{eq:perm-diagonal}
&= \frac{1}{N^2} \lambda_\emptyset^2 \left(U^\top P U\right)_{\emptyset,\emptyset}^2 
+ \frac{1}{N^2} \sum_{\substack{S = \emptyset,\; T \ne \emptyset}} \lambda_S \lambda_T \left(U^\top P U\right)_{S,T}^2
+ \frac{1}{N^2} \sum_{\substack{S \ne \emptyset,\; T = \emptyset}} \lambda_S \lambda_T \left(U^\top P U\right)_{S,T}^2 \notag \\
&\quad + \frac{1}{N^2} \sum_{\substack{S \ne \emptyset,\; T \ne \emptyset}} \lambda_S \lambda_T \left(U^\top P U\right)_{S,T}^2 \notag \\
&= \frac{1}{N^2} \lambda_\emptyset^2 
+ \frac{1}{N^2} \sum_{\substack{S \ne \emptyset, T \ne \emptyset}} \lambda_S \lambda_T \left(U^\top P U\right)_{S,T}^2 \\
&\ge \frac{1}{4} 
+ \frac{1}{N^2} \sum_{\substack{S \ne \emptyset, T \ne \emptyset}} \lambda_S \lambda_T \left(U^\top P U\right)_{S,T}^2.
\end{align*}

\noindent So, we have that over all hypercube bijections $f: \{-1, +1\}^n \rightarrow \{-1, +1\}^n$
\begin{align*}
\Pr\left[ \langle x, y \rangle \ge 0 \;\text{and}\; \langle f(x), f(y) \rangle \ge 0 \right]
&\ge \min_{P \in \mathcal{P}} \frac{1}{N^2} \operatorname{Tr}(P^\top M P M) \\
&\ge \frac{1}{4} 
+ \min_{P \in \mathcal{P}} \frac{1}{N^2} \sum_{\substack{S \ne \emptyset \\ T \ne \emptyset}} \lambda_S \lambda_T \left(U^\top P U\right)_{S,T}^2.
\end{align*}

\begin{lemma}\label{lem:rem1}
Let \( M \) be as defined above, with eigendecomposition \( M = U \Lambda U^\top \), where \( \Lambda = \operatorname{diag}(\lambda_S) \), and \( \mathcal{P} \) the set of \( N \times N \) permutation matrices. Then
\[
 \min_{P \in \mathcal{P}} \frac{1}{N^2} \sum_{\substack{S \ne \emptyset \\ T \ne \emptyset}} \lambda_S \lambda_T \left(U^\top P U\right)_{S,T}^2 \geq -o(1).
\]
\end{lemma}

\begin{proof}
    
We begin by rewriting our optimization problem as follows($\circ$ denote the Hadamard product of matrices):

\begin{gather}
\begin{aligned}
\min_{P \in \mathcal{P}} \;\; \frac{1}{N^2}
\sum_{\substack{S \ne \emptyset \\ T \ne \emptyset}} 
\lambda_S \lambda_T \left(U^\top P U\right)_{S,T}^2
&= \min_{P \in \mathcal{P}} \;\; \frac{1}{N^2}
\sum_{\substack{S \ne \emptyset \\ T \ne \emptyset}} 
\lambda_S \lambda_T \left[(U^\top P U) \circ (U^\top P U)\right]_{S,T} \\
&= \left\{
\begin{aligned}
\min_{B} \quad & \frac{1}{N^2} 
\sum_{\substack{S \ne \emptyset \\ T \ne \emptyset}} \lambda_S \lambda_T B_{S,T} \\
\text{s.t.} \quad & B \in \mathcal{B}' := 
\left\{ (U^\top P U) \circ (U^\top P U) : P \in \mathcal{P} \right\}
\end{aligned}
\right.
\end{aligned}
\label{eq:rewritten_opt}
\end{gather}

Our strategy for lower bounding this quantity is to consider a sufficient LP relaxation of \eqref{eq:rewritten_opt}. We do this by observing that $\mathcal{B}' \subseteq \mathcal{D}_N$, where $\mathcal{D}_N$ denotes the Birkhoff polytope. This follows from the fact that since both \( U \) and \( P \) are orthogonal, the matrix \( U^\top P U \) is also orthogonal. In particular, for any matrix $B \in \mathcal{B}'$, we have that for each row \( i \),
\[
\sum_{j=1}^N B_{ij} = \sum_{j=1}^N (U^\top P U)_{ij}^2 = \left\| (U^\top P U)_{i\cdot} \right\|_2^2 = 1,
\]
and similarly, for each column \( j \),
\[
\sum_{i=1}^N B_{ij} = \sum_{i=1}^N (U^\top P U)_{ij}^2 = \left\| (U^\top P U)_{\cdot j} \right\|_2^2 = 1.
\]
Each entry \( B_{ij} \ge 0 \) since it is a square. Therefore, \( B \) is nonnegative with all row and column sums equal to 1, implying \( B \in \mathcal{D}_N \). Furthermore, from \Cref{lem:matrix_fact}, we know that every matrix $B \in \mathcal{B}'$ has $B_{\emptyset, \emptyset} = 1$ and $B_{\emptyset, T} = B_{T, \emptyset} = 0$. Thus, 

\[
\mathcal{B}' \subseteq 
\left\{
\begin{bmatrix}
1 & \mathbf{0}^\top \\
\mathbf{0} & D
\end{bmatrix}
\;\middle|\;
D \in \mathcal{D}_{N-1}
\right\}
= \mathcal{D}_N^{(0)}\]

\noindent Defining $\mathcal{S} := 2^{[n]} \setminus \emptyset$, we now have the following LP relaxation-based lower bound for \eqref{eq:rewritten_opt}

\begin{align*}
\eqref{eq:rewritten_opt}
&\geq \min_{B \in \mathcal{D}_N^{(0)}} \frac{1}{N^2}
\sum_{\substack{S \subseteq [n] \\ S \ne \emptyset}} \sum_{\substack{T \subseteq [n] \\ T \ne \emptyset}} \lambda_S \lambda_T B_{S,T} \\
&= \min_{D \in \mathcal{D}_{\mathcal{S}}} \frac{1}{N^2}
\sum_{S \in \mathcal{S}} \sum_{T \in \mathcal{S}} \lambda_S \lambda_T D_{S,T}
\quad  \\
&= \min_{\pi \in \text{Sym}(\mathcal{S})} \frac{1}{N^2}
\sum_{S \in \mathcal{S}} \lambda_S \lambda_{\pi(S)}
\end{align*}
where the last equality follows from the fact that LPs achieve their objective at their vertices, and that the vertices of the Birkhoff polytope are the permutation matrices (\Cref{fact: epp} and \Cref{fact:bvn_thm}). Let $\lambda_1 \geq \lambda_2, \cdots \geq \lambda_{N - 1}$ be the ordering of the eigenvalues $\{\lambda_S\}_{S \in \mathcal{S}}$. By the rearrangement inequality, we have 
\[
\min_{\pi \in \text{Sym}(\mathcal{S})} \frac{1}{N^2}
\sum_{S \in \mathcal{S}} \lambda_S \lambda_{\pi(S)} = \frac{1}{N^2} \sum_{i = 1}^{N-1} \lambda_i\lambda_{N-i}
\]
Hence, \Cref{lem:rem1} is reduced to \Cref{lem:rem2}, which we will prove in the next section.
\begin{lemma}\label{lem:rem2}
    Let $\lambda_1 \geq \lambda_2 \geq ... \geq \lambda_{N-1}$ be the multiset $\{\lambda_S: \emptyset \neq S \subseteq 2^{[n]}\}$ sorted in descending order, here $N=2^n$.
        \[r_n := \frac{1}{N^2}\sum_{k=1}^{N}\lambda_i\lambda_{N-i}\]
        Then $r_n \geq  -o(1)$ as $n \to \infty$
\end{lemma}
\end{proof}

\section{Bounding the remainder term}
In this section, we prove \Cref{lem:rem2}, thus finishing the proof of \Cref{thm:main}. For simplicity, we assume $n=4m$ for some $m \in \mathbb{N}$. The proof works for all $n$ with minor adjustment. Our proof utilizes the following well-known facts, which we state without proof.

\begin{fact}[5.16 in \cite{graham1994}]\label{lem:altsum}
$\sum_{j=0}^{D}(-1)^{j}\binom{n}{j}=(-1)^D\binom{n-1}{D}$
    \end{fact}

\begin{fact}[Theorem IV.4 in \cite{flajolet2009}]\label{lem:coeff}
If $P(z)$ is an analytic function in a region containing the unit disk, then the coefficient of $z^k$ admits the integral expression $[z^k]P(z) = \frac{1}{2\pi}\int_{0}^{2\pi}e^{-ik\theta}P(e^{i\theta})d\theta$
    \end{fact}
The Gamma Function $\Gamma(x)$, in connection with the Beta Function $B(x,y)$, are well-studied special functions in analysis. Their definitions can be found in Chapter 8 of \cite{rudin1976}. Listed below are several facts used in our proof.
    \begin{fact}[8.21 in \cite{rudin1976}]\label{lem:trigbeta}
        $\mathrm{B}(x, y) = 2 \int_0^{\pi/2} (\sin \theta)^{2x - 1} (\cos \theta)^{2y - 1} \, d\theta$
    \end{fact}
    \begin{fact}[8.19 in \cite{rudin1976}]\label{lem:logconvex}
        $\Gamma(x)$ is log-convex. As a result, when $x+y$ is fixed, $B(x,y)=\frac{\Gamma(x)\Gamma(y)}{\Gamma(x+y)}$ is unimodal and minimizes when $x=y$.
    \end{fact}
    \begin{fact}[Section 12.33 of \cite{whittaker2021}]\label{lem:stirling}
    $\Gamma(z) = \sqrt{\frac{2\pi}{z}} \left( \frac{z}{e} \right)^z \left( 1 + O\left( \frac{1}{z} \right) \right)$
    \end{fact}
    \begin{corollary}\label{cor:asympbeta}
     \[B(x,y)=\frac{\sqrt{2\pi} \, x^{x - 1/2} y^{y - 1/2}}{(x + y)^{x + y - 1/2}}\frac{(1+O(\frac{1}{x}))(1+O(\frac{1}{y}))}{1+O(\frac{1}{x+y})}\]
    In particular, there exist constant $C>0$ such that $\frac{1}{C}\frac{\sqrt{2\pi} \, x^{x - 1/2} y^{y - 1/2}}{(x + y)^{x + y - 1/2}} \leq|B(x,y)| \leq C\frac{\sqrt{2\pi} \, x^{x - 1/2} y^{y - 1/2}}{(x + y)^{x + y - 1/2}}$ for all $x,y \geq \frac{1}{2}$
    \end{corollary}

\subsection{Explicit formulae for eigenvalues}
\begin{proposition}\label{prop:genfunc}
$\lambda_S = [x^{2m}](1-x)^{|S|-1}(1+x)^{4m-|S|}$
    \end{proposition}
    \begin{proof}
        \begin{align*}
            \lambda_S& = \sum_{k=0}^{2m} \sum_{j=0}^{k}(-1)^{j}\binom{|S|}{j}\binom{n-|S|}{k-j}\\
            &= \sum_{0 \leq j \leq k \leq 2m} (-1)^{j}\binom{|S|}{j}\binom{n-|S|}{k-j}\\
            &= \sum_{0 \leq j+k \leq 2m}(-1)^{j}\binom{|S|}{j}\binom{n-|S|}{k}\\
            &= \sum_{k=0}^{2m}\binom{n-|S|}{k} \sum_{j=0}^{2m-k}(-1)^{j}\binom{|S|}{j}\\
            &= \sum_{k=0}^{2m}\binom{n-|S|}{k} (-1)^{2m-k}\binom{|S|-1}{2m-k} \ (\Cref{lem:altsum})\\
            &= \sum_{k=0}^{2m}(-1)^k\binom{|S|-1}{k}\binom{n-|S|}{2m-k}\\
            &=[x^{2m}](1-x)^{|S|-1}(1+x)^{n-|S|}
        \end{align*}
    \end{proof}
    
    \begin{proposition}\label{prop:formula}
        \[\lambda_S = \begin{cases}
            \frac{2^{4m-1}}{\pi}(-1)^{\frac{|S|}{2}}B(\frac{|S|+1}{2},\frac{4m+1-|S|}{2}),&\ |S|\text{ even}\\
            \frac{2^{4m-1}}{\pi}(-1)^{\frac{|S|-1}{2}}B(\frac{|S|}{2},\frac{4m+2-|S|}{2}),& \ |S|\text{ odd}
        \end{cases}\]
    \end{proposition}
    \begin{proof}
        By \Cref{prop:genfunc} and \Cref{lem:coeff}
        \begin{align*}
            \lambda_S &= \frac{1}{2\pi}\int_{0}^{2\pi}e^{-2mi\theta}(1-e^{i\theta})^{|S|-1}(1+e^{i\theta})^{4m-|S|}d\theta\\
            &=\frac{1}{2\pi}\int_{0}^{2\pi}e^{-2mi\theta}(2\sin\frac{\theta}{2}e^{i(\frac{\theta}{2}-\frac{\pi}{2})})^{|S|-1}(2\cos\frac{\theta}{2}e^{i\frac{\theta}{2}})^{4m-|S|}d\theta\\
            &=\frac{2^{4m-1}}{2\pi}i^{1-|S|}\int_{0}^{2\pi}e^{-i\frac{\theta}{2}}(\sin\frac{\theta}{2})^{|S|-1}(\cos\frac{\theta}{2})^{4m-|S|}d\theta\\
            &=\frac{2^{4m-2}}{\pi}i^{1-|S|}(\int_{0}^{2\pi}(\sin\frac{\theta}{2})^{|S|-1}(\cos\frac{\theta}{2})^{4m+1-|S|}d\theta-i\int_{0}^{2\pi}(\sin\frac{\theta}{2})^{|S|}(\cos\frac{\theta}{2})^{4m-|S|}d\theta)\\
            &=\frac{2^{4m-1}}{\pi}i^{1-|S|}(\int_{0}^{\pi}(\sin\theta)^{|S|-1}(\cos\theta)^{4m+1-|S|}d\theta-i\int_{0}^{\pi}(\sin\theta)^{|S|}(\cos\theta)^{4m-|S|}d\theta)
        \end{align*}
        Note that when $|S|$ is odd, $\int_{0}^{\pi}(\sin\theta)^{|S|}(\cos\theta)^{4m-|S|}d\theta=0$, when $|S|$ is even, $\int_{0}^{\pi}(\sin\theta)^{|S|-1}(\cos\theta)^{4m+1-|S|}d\theta=0$, the result follows from \Cref{lem:trigbeta}, parity casework and symmetry.
    \end{proof}
    \begin{remark}
        When $n=4m+1$ or $4m+3$, there's a uniform expression.
    \end{remark}

    For notational convenience, we define $\lambda(|S|):=\lambda_S$, as it depends only on $|S|$.

    \begin{proposition}\label{prop:sorteigen}
    The positive eigenvalues are:
    \[\mu_1:=\lambda(1)=\lambda(4m)\]
    \[\mu_{4k+1}:=\lambda(4k)=\lambda(4k+1)=\lambda(4m-4k+1)=\lambda(4m-4k)\]
    The negative eigenvalues are:
    \[-\mu_{4k+3}:=\lambda(4k+2)=\lambda(4k+3)=\lambda(4m-4k-1)=\lambda(4m-4k-2)\]
    where
    \[\mu_{2k+1} = \frac{2^{4m-1}}{\pi}B(\frac{2k+1}{2},\frac{4m+1-2k}{2})\]
    \[\mu_1 > \mu_3 > \mu_5 > ... > \mu_{2m+1}\]
    are the distinct absolute values of eigenvalues, sorted in descending order.
    \end{proposition}
    \begin{proof}
        Consequence of \Cref{prop:formula} and \Cref{lem:logconvex}
    \end{proof}
    \begin{proposition}\label{prop:counteigen}
        \[n_{2k+1}:=|\{S\subseteq 2^{[n]}:\lambda_S = (-1)^k\mu_{2k+1}\}|\]
        Then
        \[n_1<n_3<...<n_{2m+1} \ \text{and} \ n_{2k+1}<4\binom{4m}{2k+1}\]
    \end{proposition}

\subsection{Proof of \Cref{lem:rem2}}
    By \Cref{prop:sorteigen} and \Cref{prop:counteigen}, we have
    \[-r_n \leq \frac{2}{N^2}\sum_{k=0}^{m}n_{2k+1}\mu_{2k+1}\mu_{2k+3}\]
    By \Cref{prop:formula}, \Cref{prop:sorteigen}, \Cref{prop:counteigen} and \Cref{cor:asympbeta},
    \begin{align*}
        \frac{2}{N^2}n_{2k+1}\mu_{2k+1}\mu_{2k+3}&=\frac{1}{2\pi^2}n_{2k+1}B(\frac{2k+1}{2},\frac{4m+1-2k}{2})B(\frac{2k+3}{2},\frac{4m-1-2k}{2})\\
        &\leq \frac{2}{\pi^2}\binom{4m}{2k+1}B(\frac{2k+1}{2},\frac{4m+1-2k}{2})B(\frac{2k+3}{2},\frac{4m-1-2k}{2})\\
        &= \frac{2}{\pi^2}\frac{1}{4m+1}\frac{B(\frac{2k+1}{2},\frac{4m+1-2k}{2})B(\frac{2k+3}{2},\frac{4m-1-2k}{2})}{B(2k+2,4m-2k)}\\
        &\leq \frac{2C^3}{\pi^2}\frac{1}{4m+1}\frac{\frac{(\frac{2k+1}{2})^{k}(\frac{4m+1-2k}{2})^{2m-k}}{(2m+1)^{2m+\frac{1}{2}}}\frac{(\frac{2k+3}{2})^{k+1}(\frac{4m-1-2k}{2})^{2m-k-1}}{(2m+1)^{2m+\frac{1}{2}}}}{\frac{(2k+2)^{2k+\frac{3}{2}}(4m-2k)^{4m-2k-\frac{1}{2}}}{(4m+2)^{4m+\frac{3}{2}}}}\\
        &=\frac{4C^3}{\pi^2}\frac{\sqrt{4m+2}}{4m+1}\frac{1}{\sqrt{(2k+2)(4m-2k)}}\frac{(2k+1)^{k}(2k+3)^{k+1}}{(2k+2)^{2k+1}}\\
        &\frac{(4m+1-2k)^{2m-k}(4m-1-2k)^{2m-k-1}}{(4m-2k)^{4m-2k-1}}\\
        &\leq \frac{C_1}{\sqrt{4m}}\frac{1}{2m+1}\frac{1}{\sqrt{\frac{2k+2}{4m+2}(1-\frac{2k+2}{4m+2})}}
    \end{align*}
    Therefore,
    \begin{align*}
        -r_n &\leq \frac{2}{N^2}\sum_{k=0}^{m}n_{2k+1}\mu_{2k+1}\mu_{2k+3}\\
        &\leq \frac{C_1}{\sqrt{4m}}\sum_{k=0}^{m}\frac{1}{2m+1}\frac{1}{\sqrt{\frac{2k+2}{4m+2}(1-\frac{2k+2}{4m+2})}}
    \end{align*}
    Since
    \[\lim_{m \to +\infty}\sum_{k=0}^{m}\frac{1}{2m+1}\frac{1}{\sqrt{\frac{2k+2}{4m+2}(1-\frac{2k+2}{4m+2})}} = \int_0^{\frac{1}{2}}\frac{dx}{\sqrt{x(1-x)}}=\frac{\pi}{2}\]
    We conclude that
    \[-r_n \leq \frac{C_2}{\sqrt{4m}}=\frac{C_2}{\sqrt{n}}=o(1)\]
    as desired.

%% file: acknowledgements.tex
\section{Conclusion}

In this work, we resolved the conjecture of Rob Morris, showing that for any bijection 
\(f : \{-1,1\}^n \to \{-1,1\}^n\), the probability  
\[
\Pr_{x,y \in \{-1,1\}^n}\big[ \langle x,y \rangle \ge 0 \;\text{and}\; \langle f(x), f(y) \rangle \ge 0 \big]
 \geq \frac{1}{4} - O(\frac{1}{\sqrt{n}}).\]
 Our proof proceeds by analyzing the spectrum of the Hamming association scheme, and reducing the problem to optimization over the Birkhoff polytope. Some interesting future directions and open questions are as follows:

\begin{problem}
Extend the result to the setting of \(r > 2\) bijections. That is, given any collection of $r$ bijections $f_1, f_2, \cdots, f_r$, what can we say about 
\[
\Pr_{x,y \in \{-1,1\}^n} \!\left[ \bigwedge_{i=1}^r \langle f_i(x), f_i(y) \rangle \ge 0 \right]?
\]
\end{problem}
Our methodology completely resolves the $r = 2$ case, but for $r > 2$, it is not clear how to extend it. In particular, it is not too difficult to generalize our method to the $r$-variable case and attain the lower bound
\[
\Pr_{x,y \in \{-1,1\}^n} \!\left[ \bigwedge_{i=1}^r \langle f_i(x), f_i(y) \rangle \ge 0 \right] \geq \bigg(\frac{1}{2}\bigg)^r + \frac{1}{N^r}\min_{X \in \mathcal{D}^{(r)}_N} \langle X, \lambda^{\otimes r} \rangle
\]
where $\lambda \in \mathbb{R}^N$ has its $i^{th}$ entry corresponding to the $i^{th}$ largest eigenvalue of $M$ and $\mathcal{D}^{(r)}_N$ is the set of tensors given by \[
\mathcal{D}^{(r)}_N \;=\; 
\Bigl\{ \, T \in \mathbb{R}_{\ge 0}^{N^r} \;:\;
\sum_{j=1}^N T_{i_1,\dots,i_{k-1},j,i_{k+1},\dots,i_r} = 1 
\;\;\text{for all } k \in [r],\;\; (i_1,\dots,i_{k-1},i_{k+1},\dots,i_r) \in [N]^{r-1} \,\Bigr\}, 
\]
which is the natural extension of doubly stochastic matrices to higher-order tensors, requiring that every axis-parallel line sums to \(1\). The main difficulty lies in lower bounding the linear program  
\[
\frac{1}{N^r}\min_{X \in \mathcal{D}^{(r)}_N} \langle X, \lambda^{\otimes r} \rangle,
\]  
since for \(r > 3\) no analytic solution is available. Numerical evidence suggests that \(-o(1)\) remains a valid lower bound, but a rigorous proof is still missing. Moreover, while the tensor structure of the LP ensures—via a standard cycling argument—the existence of an integral solution, the absence of a suitable ``high-dimensional rearrangement inequality'' prevents us from characterizing the objective explicitly.
\\\\
Another related research direction (proposed by Rob Morris) is the following:
\\\\
\begin{problem}
Can we characterize subsets \(A,B \subseteq \{0,1\}^n\) for which analogous lower bounds continue to hold? In particular, the lemma in \cite{balister2024upperboundsmulticolourramsey} implies that if \(f : A \to B\) is a bijection and 
\[
\Pr\!\big[ \langle x,y \rangle \ge 0 \;\text{and}\; \langle f(x),f(y)\rangle \ge 0 \big] = o(1),
\]
then one of \(A\) or \(B\) must be ``clustered.'' Can this condition be weakened to only requiring the probability to be less than \(1/4 - c\) for some constant \(c > 0\)?
\end{problem}

\section{Acknowledgements}

This work was completed at PCMI 2025. As such, we would like to thank the organizers and many individuals involved with running the program. In particular, we would like to thank Noga Alon and Yuval Wigderson for insightful discussions regarding this problem.